\documentclass{amsart}
\usepackage{marktext}
\usepackage{fy}

\usepackage[tight]{subfigure}

\colorlet{darkred}{red!40!black}

\newdelim{\ip}{\langle}{\rangle}

\makeatletter
\newcommand{\leqnomode}{\tagsleft@true\let\veqno\@@leqno}
\newcommand{\reqnomode}{\tagsleft@false\let\veqno\@@eqno}
\newcommand{\mylabel}[2]{\def\@currentlabel{#2}\label{#1}}
\makeatother

\newcommand{\FL}{\mathcal{L}}

\begin{document}
\title[Dissipation enhancement of cellular flows]%
  {Dissipation enhancement of cellular flows in general advection diffusion equations}
\author[Feng]{Yu Feng}
\address{%
  Yu Feng, Beijing International Center for Mathematical Research, Peking University, No. 5 Yiheyuan Road Haidian District, Beijing, P.R.China 100871}
\email{fengyu@bicmr.pku.edu.cn}
\author[Xu]{Xiaoqian Xu}
\address{Xiaoqian Xu: Zu Chongzhi Center for Mathematics and Computational Sciences, Duke Kunshan University, China}
\email{xiaoqian.xu@dukekunshan.edu.cn}
\date{\today}

\begin{abstract}
The main contribution of this paper is twofold: (1) Recently, Iyer, Xu, and Zlatoš studied the dissipation enhancement by cellular flows based on standard advection-diffusion equations via a stochastic method. We generalize their results to advective hyper-diffusion equations and advective nonlinear diffusion equations. (2) We prove there exist smooth incompressible flows that are relaxation enhancing to hyper-diffusion but not to standard diffusion.
\end{abstract}
\maketitle

\section{Introduction}
Let $\T^d=[0,1]^d$ denotes the $d$ dimensional torus, $L_0^2$ be the mean-zero subspace of $L^2(\T^d)$ with the inner product $\langle\cdot ,\cdot\rangle$, and $u$ be a divergence-free vector field. Then we consider the following equation on $\T^d$:
\begin{equation}
\label{e:gen-adv-diff-eqn}
    \partial_t\theta + u\cdot\nabla\theta+\gamma\left(-\lap\right)^{\alpha}\theta=0,
\end{equation}
with $\alpha,\gamma>0$, and initial data $\theta(0,x)=\theta_0\in L_0^2$. In particular, for $\alpha=1$, it gives the standard advection-diffusion equation
\begin{equation}
\label{e:adv-diff-eqn}
    \partial_t\theta + u\cdot\nabla\theta-\gamma\lap\theta=0.
\end{equation}
For $\alpha=2$, \eqref{e:gen-adv-diff-eqn} becomes the advective hyper-diffusion equation
\begin{equation}
\label{e:adv-hyperdiff-eqn}
    \partial_t\theta + u\cdot\nabla\theta+\gamma\lap^2\theta=0.
\end{equation}
The dissipation enhancement effect has been studied extensively in recent decades. Most of these research is carried out based on \eqref{e:adv-diff-eqn}, and can be generalized to \eqref{e:gen-adv-diff-eqn} for $\alpha>1$ without any difficulty. However, the results obtained from the maximal principle or stochastic process associated with \eqref{e:adv-diff-eqn} cannot be generalized to \eqref{e:gen-adv-diff-eqn} trivially (for example \cites{fannjiang2006quenching,novikov2007boundary,iyer2021convection,iyer2022quantifying}). In this paper, we show the dissipation enhancement result obtained via a stochastic method in \cite{iyer2021convection} can be generalized to \eqref{e:gen-adv-diff-eqn} and nonlinear diffusion equations, in which the stochastic structure is lost. 

To begin with, we define the \emph{dissipation time} associate to \eqref{e:gen-adv-diff-eqn}.
\begin{definition}
\label{def:dissipationTime}
Let $\mathcal S_{s, t}^{u,\alpha}$ be the solution operator to \eqref{e:gen-adv-diff-eqn} on $(0,\infty)\times\T^d$, that is, for any function $\theta_s(x)\in L^2_0$, $S_{s,t}^{u,\alpha}\theta_s(x)$ solves \eqref{e:gen-adv-diff-eqn} with initial data $\theta(s,x)=\theta_s(x)$ on $\mathbb{T}^d$. The \emph{dissipation time} of $u$ is defined by
  \begin{equation}
  \label{e:dtimeDef}
    \tau_{\alpha}^*(u,\gamma) \defeq  \inf \set[\Big]{ t \geq 0 \st  \norm{ \mathcal S_{s, s+t}^{u,\alpha} }_{L_0^2 \to L_0^2 } \leq \tfrac{1}{2} \text{ for all $s\geq 0$} }\,.
  \end{equation}
\end{definition}
Recent researches indicate that the flows with small dissipation time can be used to suppress the formation of singularities in the nonlinear PDEs (see \cites{iyer2021convection,feng2020phase,feng2022global,feng2022suppression}). In the second-order differential equations, such as the Keller-Segel equations and the ignition type reaction-diffusion equations, one needs to provide flows such that $\tau_1^{*}(u,\gamma)$ is small enough \cites{iyer2021convection}. And for the fourth order differential equations, such as Cahn-Hilliard equation \cite{feng2020phase}, Kuramoto-Sivashinsky equation \cite{feng2022global} and thin-film equation \cite{feng2022suppression}, one need to provide flows such that $\tau_2^{*}$ is sufficiently small (see Section \ref{sec:Applications to reaction hyper-diffusion equations} for more details). Therefore, it is meaningful to construct flows with arbitrarily small dissipation times for either $\alpha=1$ or $2$.

In the seminal paper \cite{ConstantinKiselevEA08}, the main result interpret that given a time-independent velocity field $u$, $\tau^{*}_{1}(Au,\gamma)$ tends to zero as $A$ tends to infinity if and only if the operator $u\cdot\nabla$ has no eigenfunctions in $H^{1}(\T^d)$ rather than constants. A typical class of such flows is the weakly mixing flows, for which $u\cdot\nabla$ only has continuous spectrums. The proof of \cite{ConstantinKiselevEA08} relies on the so-called RAGE theorem, as a consequence, contains no quantitative information on the appearance of a faster time scale induced by $u$. After that, the authors in \cites{FengIyer19,CotiZelatiDelgadinoEA18} obtain explicit 
bounds on $\tau_1^{*}(u,\gamma)$, by requiring the \emph{mixing rate} of a mixing flow $u$, associated to the underlying transport equation (see \cite{thiffeault2012using}). 
% In order to get a quantitative estimate of the diffusion time, the authors in \cite{FengIyer19,CotiZelatiDelgadinoEA18} consider the re-scaling of the mixing flows. 
Here the intuition is based on the fact that the mixing flows create small-scale structures (or high frequencies). And the faster this process happens, the faster the energy of solution to \eqref{e:gen-adv-diff-eqn} is dissipated and, as a result, the smaller $\tau_{1}^*(u,\gamma)$ becomes. However, known examples of such flows are either complicated or not regular enough. 

Fortunately, many smooth flows are not mixing but have small dissipation times, thus enough to suppress singularities in nonlinear PDEs. For such flows, the dissipation times do not tend to zero as one increases the amplitude of the flow. One typical class of such flows is cellular flows (see \cites{iyer2021convection,kiselev2001enhancement} for applications). A prototypical example of cellular flow in two dimensions is given by
\begin{equation}\label{eq:cell}
    u(x)=\grad^{\perp}\sin(2\pi x_1)\sin(2\pi x_2).
\end{equation}
Similar to the example above, all two-dimensional cellular flows have closed trajectories and hence are not mixing. In \cite{iyer2021convection}, based on equation \eqref{e:adv-diff-eqn}, the authors proved that the dissipation time, $\tau^*_1$, of such flows could be made arbitrarily small by rescaling the spacial scale and the flow amplitude at the same time. Their proof is based on establishing the relationship between the dissipation time $\tau_1(u,1)$ and the \emph{effective diffusivity}, denote as $D(u)$ (see Section \ref{sec:pre} for a more precise definition). For convenience to the reader, we cite their theorem here.
\begin{theorem}\label{thm:cellular for tau1}
For each $m\in \N$, let $u_m\in W^{1,\infty}(\T^d)$ be a mean zero divergence free velocity field that is symmetric (condition \eqref{e:symmetric condition} holds) in all coordinates. If the effective diffusivity $D(u_m)$ satisfies $\lim\limits_{m\rightarrow\infty} D(u_m)=\infty$, then there exists $l_m\in \mathbb{N}$, such that the rescaled velocity fields $v_m(x)\defeq-l_m u_m(l_m x)$ satisfy
\begin{equation}\label{eq:difftime}
    \lim_{m\rightarrow\infty}\tau_1^{*}(v_m,\gamma)=0.
\end{equation}

In particular, for $d\in\left\{2,3\right\}$ and any $\sigma>0$, there exists a smooth cellular flow $u$ on $\T^d$ such that $\tau_1^{*}(u,\gamma)\leqslant \sigma$.
\end{theorem}
\begin{remark}
\label{rmk:23dimension dissipation}
Based on the example $(5.9)$ in \cite{iyer2021convection}, for instance, when $d=2$, one can choose $v_m=m^{2+\alpha}u(mx)$, with $\alpha=\frac{15}{113}$, where $u(x)$ is given by \eqref{eq:cell}, so that \eqref{eq:difftime} holds. Similar examples exist in three dimensions. In the recent paper \cite{iyer2022quantifying},  by constructing a successful coupling, Iyer and Zhou improve this result by providing quantitative bounds on the dissipation enhancement in terms of the flow amplitude, cell size, and diffusivity.
\end{remark}
Both the proofs in \cite{iyer2021convection} and \cite{iyer2022quantifying} rely on the stochastic processes associated to equation \eqref{e:adv-diff-eqn}, and therefore cannot be generalized to \eqref{e:gen-adv-diff-eqn} directly. On the other hand, in \cite{feng2020phase} the authors derived an upper bound of $\tau_2^{*}(u,\gamma)$ in terms of $\tau_1^{*}(u,\gamma)$ for any smooth divergence-free vector field. We cite their result here.
\begin{lemma}\label{lem:relation btw tau_1and2}
There exists an explicit dimensional constant $C$ such that for every divergence free $u\in L^{\infty}([0,\infty);C^2(\T^d))$, and every $\gamma>0$, we have
\begin{equation}\label{eqn:relation btw tau_1and2}
    \tau_2^{*}(u,\gamma)\leqslant C\tau_1^{*}(u,\gamma)(1+\norm{u}_{C^2}\tau_1^{*}(u,\gamma)).
\end{equation}
\end{lemma}
Once velocity fields with small $\tau_1^{*}(u,\gamma)$ are known, one can use Lemma \ref{lem:relation btw tau_1and2} to produce velocity fields for which $\tau_2^{*}(u,\gamma)$ is small. For instance, if $u$ is mixing at a sufficiently fast rate, then the results of Wei\cite{wei2021diffusion}, Coti Zelati et al.\cite{CotiZelatiDelgadinoEA18}, Feng and Iyer\cite{FengIyer19} along with Lemma \ref{lem:relation btw tau_1and2} can be used to produce velocity fields for which $\tau_2^{*}(u,\gamma)$ is arbitrarily small. However, Lemma \ref{lem:relation btw tau_1and2} is not tight enough to produce arbitrary small $\tau_2^{*}(u,\gamma)$ for cellular flows. Indeed, with the $\tau_1^*$ bound of cellular flows in \cite{iyer2021convection} and \cite{iyer2022quantifying}, the right-hand side of \eqref{eqn:relation btw tau_1and2} blows up as $\norm{u}_{C^2}$ increases.

There are two main contributions of this paper. Firstly, we show that cellular flows indeed can be used to produce $\tau_{\alpha}^{*}(u,\gamma)$ arbitrarily small for any $\alpha\geqslant 1$, that is Theorem \ref{thm:cellular for tau1} remains true if we replace $\tau_1^*$ by $\tau_{\alpha}^*$ with $\alpha>1$, the formal conclusion is established in Theorem \ref{thm:main}. As a byproduct, we also show that cellular flows can be used to enhance dissipation in nonlinear diffusion equations, such as porous medium equations and p-Laplacian equations (see Theorem \ref{e:adv p-laplacian} and Theorem \ref{thm:RH for plap}). In the second part of this paper, we prove that there exist smooth incompressible flows that are relaxation enhancing with respect to the advective hyper-diffusion equation but not to the standard advection-diffusion equation (Proposition \ref{prop:RH flows}).

\section{Preliminary}\label{sec:pre}
We devote this section to introduce the notations and establish some elementary estimates.

For a function $f(x)$ in defined in $\T^d$, we define the homogeneous $L^2$-Sobolev norm $\dot{H}^s$ ($s\neq 0$) to be
$$
\|f\|_{\dot{H}^s}=\left(\sum\limits_{|n|\neq 0}|\hat{f}(n)|^2|n|^{2s}\right)^{\frac{1}{2}},
$$
where $\hat{f}(n)$ is the $n^{\text{th}}$ Fourier coefficient of $f$.

We say that a flow $u$ is \emph{symmetric} in $x_n$ if we have
\begin{equation}
\label{e:symmetric condition}
    u(R_n x)=R_n u(x)\qquad\text{for all }x\in\T^d.
\end{equation}
Here $R_n v\defeq v-2v_n e_n$ for $v=\sum\limits_{n=1}^{d}v_n e_n\in\R^d$ is the reflection in the $n^{\text{th}}$ coordinate. A periodic flow that is symmetric in all $d$ coordinates is said to have a cellular structure and we call it a cellular flow.

Consider \eqref{e:adv-diff-eqn} on $\R^d$, with initial data $\theta_0$, and let $u$ be time-independent, mean-zero, divergence-free, and Lipschitz continuous vector fields. For simplicity, we assume $\gamma=1$ from now on. Define the stochastic process $X_t^x(\omega)$:
\begin{equation}\label{e:sde}
    dX_t^x=\sqrt{2}dB_t-u(X_t^x)dt,\quad X_0^x=x,
\end{equation}
where $B_t$ is a normalized Brownian motion on $\R^d$ with $B_0=0$, defined on some probability space $(\Omega,\mathcal{B}_{\infty},\mathbb{P}_{\Omega})$. Let $G_t(x,y)$ denotes the fundamental solution to \eqref{e:adv-diff-eqn} and $\mathbb{E}_{\Omega}$ the expectation with respect to $\omega\in\Omega$, then the corresponding solution to \eqref{e:adv-diff-eqn} can be written as $\theta(t,x)=\int_{\R^d}G_t(x,y)\theta_0(y)dy=\mathbb{E}_{\Omega}(\theta_0(X_t^x))$. Then we define the $\emph{effective diffusivity}$ of $u$ in the direction $e$ by
\begin{equation}
    D_e(u)\defeq\lim_{t\rightarrow\infty}\mathbb{E}_{\Omega}\left(\frac{\abs{\left(X_t^x-x\right)\cdot e}^2}{2t}\right)\qquad(\geqslant 1),
\end{equation}
with the limit being independent of $x\in\R^d$. In addition, let 
\begin{equation}
    D(u)\defeq\min\left\{D_{e_1}(u),\dots,D_{e_d}(u)\right\}\geqslant 1
\end{equation}
denotes the minimum of effective diffusivities in all the coordinate directions (see \cite{Zlatos11} for more details). Intuitively, $D_e(u)$ is the average distance a point can go along with the process \eqref{e:sde}.

Let $\FL_m\defeq iu(m x)\cdot\grad$ with $u$ divergence free, note that $\FL_m$ is a self-adjoint operator with respect to $\langle\cdot ,\cdot\rangle$. Let $P_c$ and $P_p$ be the spectral projection on its continuous and discrete spectral subspace correspondingly. On the other hand, we denote by $0<\lambda_1\leqslant\lambda_2\leqslant\dots$ the eigenvalues of the operator $-\lap$ on the torus, and $P_N$ be the orthogonal projection on the subspace spanned by its first $N$ eigen-modes, and by $S=\left\{\theta\in L^2:\norm{\theta}_{L^2}=1\right\}$ the unit sphere in $L^2$.
For the convection field with an amplitude $A>0$,  equation \eqref{e:gen-adv-diff-eqn} can be represented in the following way:
\begin{equation}
\label{e:Aform1.1}
    \frac{d}{dt}\theta_m^{A}(t)=i A\FL_m\theta_m^{A}(t)-(-\lap)^{\alpha}\theta_m^A(t),
\end{equation}
with initial date $\theta_m^{A}(0)=\theta_0$. Here we assume $\alpha\geq 1$. For convenient, we rescale \eqref{e:Aform1.1} in time by the factor $\epsilon^{-1}=A$, thus we obtain the following equivalent reformulation:
\begin{equation}
\label{e:eform1.1}
    \frac{d}{dt}\theta_m^{\epsilon}(t)=i \FL_m\theta_m^{\epsilon}(t)-\epsilon(-\lap)^{\alpha}\theta_m^{\epsilon}(t),\qquad \theta_m^{\epsilon}(0)=\theta_0.
\end{equation} 
Multiplying $\theta^{\epsilon}_m(t)$ and integration by parts in $x$, one can easily get
\begin{equation}
\label{e:L2 rate}
    \frac{d}{dt}\norm{\theta_m^{\epsilon}}_{L^2}^2=-2\epsilon\norm{\theta_m^{\epsilon}}_{\dot{H}^{\alpha}}^2,
\end{equation}
and we have the following lemma on the decay of $\|\theta_m^{\epsilon}\|_{L^2}$.
\begin{lemma}
\label{lem:trivial case}
For any fixed $m>0$, suppose that for all times $t\in(a,b)$ we have $\norm{\theta_m^{\epsilon}(t)}_{\dot{H}^{\alpha}}^2\geqslant N\norm{\theta_m^{\epsilon}(t)}_{L^2}^2$. Then the following decay estimate holds
\begin{equation}
    \norm{\theta_m^{\epsilon}(b)}_{L^2}^2\leqslant e^{-2\epsilon N(b-a)}\norm{\theta_m^{\epsilon}(a)}_{L^2}^2.
\end{equation}
\end{lemma}
We also need an estimate on the growth of the difference between the viscous and inviscid problem in terms of $L^2$ norm. First notice that for any smooth incompressible flow $u(m x)$, and define $\FL_m$ as before, then for any $\phi\in\dot{H}^{\alpha}$, $\alpha\geq 1$ and $t>0$ the following estimates hold:
\begin{equation}
\label{e:F_m condition}
    \norm{\FL_m\phi}_{L^2}\leqslant C\|u(m\cdot)\|_{L^{\infty}}\norm{\phi}_{\dot{H}^{\alpha}}\quad\text{and}\quad\norm{e^{i\FL_m t}\phi}_{\dot{H}^{\alpha}}\leqslant F_m(t)\norm{\phi}_{\dot{H}^{\alpha}}
\end{equation}
with both the constant $C$ and the function $F_m(t)<\infty$ independent of $\phi$ and $F_m(t)\in L_{\text{loc}}^2(0,\infty)$. Here $F_m(t)$ depends on $\|u(m\cdot)\|_{H^{\alpha}}$.
\begin{lemma}
\label{lem:difference estimate}
Fix any $m>0$, let $\theta_m^{0}(t)$, $\theta_m^{\epsilon}(t)$ be solution of 
\begin{equation}
\label{e:two equations}
    \frac{d}{dt}\theta_m^0(t)=i\FL_m\theta_m^{0}(t),\qquad\frac{d}{dt}\theta_m^{\epsilon}(t)=(i\FL_m-\epsilon(-\lap)^{\alpha})\theta_{m}^{\epsilon}(t),
\end{equation}
with $\theta_m^0(0)=\theta_m^{\epsilon}(0)=\theta_0\in\dot{H}^{\alpha}$. Then
\begin{equation}
    \frac{d}{dt}\norm{\theta_m^0(t)-\theta_m^{\epsilon}(t)}_{L^2}^2\leqslant\frac{1}{2}\epsilon\norm{\theta_m^0(t)}_{\dot{H}^{\alpha}}^2\leqslant\frac{1}{2}\epsilon F_m^2(t)\norm{\theta_0}_{\dot{H}^{\alpha}}^2.
\end{equation}
And further,
\begin{equation}
    \norm{\theta_m^0(t)-\theta_m^{\epsilon}(t)}_{L^2}^2\leqslant\frac{1}{2}\epsilon\norm{\theta_0}_{\dot{H}^{\alpha}}^2\int_0^{\tau}F_m^2(t)dt
\end{equation}
for any time $t<\tau$.
\end{lemma}
The proofs of Lemma \ref{lem:trivial case} and Lemma \ref{lem:difference estimate} are standard, the reader can find them on \cite{ConstantinKiselevEA08}.
\section{Discussion of the main theorem}
\begin{theorem}
\label{thm:main}
Let $u\in C^{\infty}(\T^2)$ be a cellular flow given in \eqref{eq:cell}. Fix $\gamma>0$, $\alpha\geq 1$. For any $\epsilon>0$, there exists $m_0>0$, such that for any $m\geq m_0$, there exists an $A_0>0$, such that $\tau_{\alpha}^{*}(Au(m\cdot),\gamma)\leq \epsilon$ for any $A\geq A_0$. 
\end{theorem}
\begin{remark}
For $d=3$, one can also have analogous of \eqref{eq:cell}, say, Example 5.10 in \cite{iyer2021convection}.
\end{remark}
\begin{remark}
The case $\alpha=\gamma=1$ has already been proved in \cite{iyer2021convection} via a stochastic method. Therefore, we only remain to show that the smallness of $\tau_1^{*}(Au(m\cdot),\gamma)$ implies the smallness of $\tau_{\alpha}^{*}(Au(m\cdot),\gamma)$ for any $\alpha>1$.
\end{remark}
\begin{remark}
Through this chapter, $\alpha\geq 1$ is a fixed number. The constant on the estimates may always depend on the value of $\alpha$.
\end{remark}
To prove the main theorem, we first establish and prove some lemmas. 

First, let us consider the eigenfunctions of $\FL_m$. We will show that the eigenfunctions of $\FL_m$ will have big $\dot{H}^{\alpha}$ norm. In fact, for the cellular flow in dimension two, the eigenfunction always exists and is of class $H^{\alpha}$. However, for the sake of simplicity we ignore the existence discussion here.

Let $\phi_m^0$ be any normalized ($\norm{\phi_m^0}_{L^2}=1$) eigenfunction of $\FL_m$ that belongs to $\dot{H}^{\alpha}$ (with $\alpha\geqslant1$) associated with eigenvalue $E_m$. Then the following lemma shows that for fixed sufficiently large $m$, the normalized point spectrums of $\FL_m$ have an uniform lower bound in terms of the $\dot{H}^{\alpha}$ norm, otherwise it contradicts with the result for $\alpha=\gamma=1$ proved in \cite{iyer2021convection}. 
\begin{lemma}
\label{lem:large_H1}
With the same assumption of $u(m\cdot)$ in Theorem \ref{thm:main}. Given $\tau_0>0$, there exists $m_0(\tau_0)$ such that for any fixed $m>m_0(\tau_0)$ it holds $\norm{\phi_m^0}_{\dot{H}^1}^2>\frac{1}{2\tau_0}$. If $\phi_m^0$ further belongs to $H^{\alpha}$ for some $\alpha>1$, then $\norm{\phi_m^0}_{\dot{H}^{\alpha}}^2>\frac{1}{2\lambda_1^{\alpha-1} \tau_0}$.
\end{lemma}
\begin{proof}
The proof is very similar to the proof of Theorem 1.4 (the easier direction) in \cite{ConstantinKiselevEA08}. We prove by showing contradiction, assume for arbitrary large $m_0$ and $m>m_0$, there exists an eigenvalue $E_m$ and corresponding eigenfunction $\phi_m^0$ (associate to operator $\FL_m$) such that $\tau_0\norm{\phi_m^0}_{\dot{H}^1}^2\leqslant\frac{1}{2}$. Consider the solution $\theta^A_{m}(t)$ to \eqref{e:adv-diff-eqn}, with convection term $A u(m x)\cdot\nabla$ and initial data $\phi_m^0\in L^2_0\cap \dot{H}^1$(with $\norm{\phi_m^0}_{L^2}=1$). Take the $L^2$ inner product of \eqref{e:adv-diff-eqn} with $\phi_m^0$, we get
\begin{equation}
  \frac{d}{dt}\langle\theta^A_m(t), \phi_m^0\rangle=iA E_m\langle\theta^A_m, \phi_m^0\rangle+\langle\lap\theta^A_m,\phi_m^0\rangle.
\end{equation}
This further yields
\begin{equation}
    \Big\vert\frac{d}{dt}\left(e^{-iA E_m t}\langle\theta^A_m, \phi_m^0\rangle\right)\Big\vert\leqslant\frac{1}{2}\left(\norm{\theta^A_m(t)}_{\dot{H}^1}^2+\norm{\phi_m^0}_{\dot{H}^1}^2\right).
\end{equation}
Note that $\int_0^{\infty}\norm{\theta^A_m(t)}_{\dot{H}^1}^2dt\leqslant\frac{1}{2}\norm{\phi_m^0}_{L^2}^2=\frac{1}{2}$ and $\tau_0\|\phi_m^0\|_{\dot{H}^1}^2\leq \frac{1}{2}$. Hence,
\begin{align*}
\left| \left(e^{-iA E_m \tau_0}\langle\theta^A_m,  \phi_m^0\rangle\right)-1\right|&\leq\int_0^{\tau_0}\left|\frac{d}{dt}\left(e^{-iA E_m t}\langle\theta^A_m,  \phi_m^0\rangle\right)\right|dt\\
&\leq \frac{1}{2}\left(\int_0^{\tau_0}\norm{\theta^A_m(t)}_{\dot{H}^1}^2dt+\tau_0\norm{\phi_m^0}_{\dot{H}^1}^2\right)\leq \frac{1}{2}.
\end{align*}
Therefore for $0<t\leqslant\tau_0$, we have $\abs{\langle\theta^A_m(t),\phi_m^0\rangle}\geq\frac{1}{2}$, which further implies $\norm{\theta^A_m(\tau_0)}_{L^2}\geq\frac{1}{2}\norm{\phi_m^0}_{L^2}$ uniformly in $A$. Equivalently, for any $m>0$ and $A>0$, the dissipation time $\tau_1^{*}(A u(m x))\geq\tau_0$, which contradicts to the Theorem \ref{thm:cellular for tau1} and Remark \ref{rmk:23dimension dissipation}. This completes the proof of the first part of the lemma. The second part directly follows from the Poincare inequality.
\end{proof}
With the help of Lemma \ref{lem:large_H1}, we can further control from below the growth of $\dot{H}^{\alpha}$ norm of solutions, to the underlying inviscid problem, corresponding to discrete eigenfunctions.
\begin{lemma}
\label{lem:H1 mean large}
With the same assumption of $u(mx)$ in Theorem \ref{thm:main}. Let $K\subset S=\{\theta\in L^2: \|\theta\|_{L^2}=1\}$ be a compact set in $L^2$. Consider the set $K_1\defeq\left\{\theta\in K\vert\norm{P_p\theta}_{L^2}\geqslant\frac{1}{2}\right\}$. Then for any $B>0$, there exists $m_0(B)$ such that for any fixed $m>m_0$ we can find $N_p(B,K,m)$ and $T_p(B,K,m)$ such that for any $N\geqslant N_p(B,K,m)$, $T\geqslant T_p(B,K,m)$ and any $\theta\in K_1\cap\dot{H}^{\alpha}$,
\begin{equation}
    \frac{1}{T}\int_0^T\norm{P_N e^{i\FL_m t} P_p\theta}_{\dot{H}^{\alpha}}^2 dt\geqslant B.
\end{equation}
\end{lemma}
\begin{proof}
The proof is quite similar to the proof of Lemma 3.3 in \cite{ConstantinKiselevEA08}. Denote by $E_j^m$ the eigenvalues of $\FL_m$(distinct, without repetitions) and by $Q_j^m$ the orthogonal projection on the space spanned by the eigenfunctions corresponding to $E_j^m$. Without lost of generality, we assume $K_1$ is nonempty, otherwise there is nothing to prove. Observe that, by applying Lemma \ref{lem:large_H1} with $\tau_0=\frac{1}{8B\lambda_1^{\alpha-1}}$ there is $m_0(B)$ such that for any $m>m_0$ we have
\begin{equation}
    \sum_j\norm{Q_j^m\theta}_{\dot{H}^{\alpha}}^2>4B.
\end{equation}
Therefore by compactness of $K$, there exists $N_p=N_p(B,K,m)$, uniform in $j$, such that for $N>N_p(B,K,m)$ 
$$
 \sum_j\norm{P_N Q_j^m\theta}_{\dot{H}^{\alpha}}^2\geq 2B.
$$
The remains proof is identical the same as the proof of Lemma 3.3 in \cite{ConstantinKiselevEA08}, with $\norm{\cdot}_1$ (defined in \cite{ConstantinKiselevEA08}) replaced by $\norm{\cdot}_{\dot{H}^{\alpha}}$.
\end{proof}
On the other hand, for the continuous spectrum of $\FL_m$ we can control it by the following RAGE type theorem.
\begin{lemma}\label{lem:RAGE thm}
Let $K\subset S$ be a compact set. For any $N,\sigma,m>0$, there exists $T_c(N,\sigma,m,K)$ such that for all $T\geqslant T_c(N,\sigma,m,K)$ and any $\theta\in K$,
\begin{equation}
    \frac{1}{T}\int_0^T \norm{P_N e^{i\FL_m t}P_c\theta}_{L^2}^2 dt\leqslant\sigma.
\end{equation}
\end{lemma}
The proof of Lemma \ref{lem:RAGE thm} can be found in \cite{ConstantinKiselevEA08}.

Now, we turn to the main theorem. We may decomposite the proof into two parts:

(1). For the iniviscid problem, by a decomposition of the initial value $\theta_0$ into continuous spectrum and discrete spectrum of $\FL_m$, one can use Lemma \ref{lem:H1 mean large} and \ref{lem:RAGE thm} lower bound of high frequencies in the sense of time average. 

(2). By Lemma \ref{lem:large_H1} and rescaling of time, one can control the difference of the solutions to the inviscid problem and the viscid problem, and based on Lemma \ref{lem:trivial case} to get the desired decay estimate. 

This argument is similar to the proof of the Theorem 1.4 in \cite{ConstantinKiselevEA08}. For completion, we provide the whole proof in Appendix.

\begin{remark}
The work in \cite{iyer2022quantifying} can be generalized to hyper-diffusion in the same manner. However, due to applying the RAGE theorem in the proof of Theorem \ref{thm:main} (or Lemma \ref{lem:RAGE thm}), the quantitative estimates in \cite{iyer2022quantifying} will be lost while translating to the hyper-diffusion scenario.
\end{remark}
% The proof of the main theorem is parallel to that of Theorem 1.4 in \cite{ConstantinKiselevEA08}, so we just indicate the necessary changes. Let $\theta^{\alpha}_{\mathcal{A},m}$ defines the solution of \eqref{e:gen-adv-diff-eqn} associated with $u=\mathcal{A} u_m(A_mx)$ and arbitrary initial data that belongs to $S$. Observe that the $\norm{\cdot}_1$ norm defined in \cite{ConstantinKiselevEA08} stands for $\norm{\cdot}_{\dot{H}^{\alpha}}$ correspond to \eqref{e:gen-adv-diff-eqn}. Then for any given $\tau>0$ and $\delta=\frac{1}{2}$ we can pick up $\mathcal{M}$ large enough such that $e^{-\lambda_{\mathcal{M}}\tau/80}<\delta$. For any fixed large $m$ (determined by Lemma \ref{lem:H1 mean large} and the value of $\lambda_{\mathcal{M}}$) we can replace Lemma 3.3 in \cite{ConstantinKiselevEA08} by Lemma \ref{lem:H1 mean large} above. Thus by modifying the proof of Theorem 1.4 in \cite{ConstantinKiselevEA08} line by line, we can prove that for $m$ sufficiently large, there exists $\mathcal{A}$ large enough such that $\norm{\theta_{\mathcal{A},m}^{\alpha}(\tau)}<1/2$, i.e. $\tau_{\alpha}^{*}(\mathcal{A} u_m(A_mx))<\tau$. This completes the proof of the first part of the theorem. For the second part of the theorem, the concrete examples in dimension two and three, we point readers to Example 5.7 and Example 5.8 in \cite{iyer2021convection}.

\section{Applications}
\subsection{Applications to reaction hyper-diffusion equations}
\label{sec:Applications to reaction hyper-diffusion equations}
In \cite{iyer2021convection}, the authors provide two applications that rely on the smallness of $\tau_1^*(u)$, one is the suppression of blow-up in the Keller-Segel system, the other one is the quenching in models of combustion. Here we collect several results that rely on the smallness of $\tau_2^*(u)$. In this chapter, we will use $\bar{f}$ to denote the mean of function $f$.
\begin{enumerate}
    \item Advective Cahn-Hilliard equation (see \cite{feng2020phase})
    \begin{equation}\label{e:CHE}
        c_t+ u\cdot\grad c + \gamma \lap^2 c = \lap(c^3-c).
    \end{equation}
    \begin{theorem}
    \label{thm:Cahn-Hilliard}
    Let $d\in\{2,3\}$, $u\in L^{\infty}([0,\infty);W^{1,\infty}(\T^d))$, and $c$ be the strong solution of \eqref{e:CHE} with initial data $c_0\in H^2(\T^d)$. 
    \begin{enumerate}
        \item When $d=2$, for any $\beta>1,\mu>0$, there exists a time
        \begin{equation*}
        T_0=T_0(\norm{c_0-\bar{c}}_{L^2},\bar{c},\beta,\gamma,\mu)
        \end{equation*}
        such that if $\tau_2^{*}(u,\gamma)<T_0$, then for every $t\geq 0$, we have
        \begin{equation}\label{e:CHE-exp-decay}
            \norm{c(t)-\bar{c}}_{L^2}\leqslant\beta e^{-\mu t}\norm{c_0-\bar{c}}_{L^2}.
        \end{equation}
        \item When $d=3$, for any $\beta>1,\mu>0$, there exists a time
        \begin{equation*}
            T_1=T_1(\norm{c_0-\bar{c}}_{L^2},\bar{c},\beta,\gamma,\mu)
        \end{equation*}
        such that if
        \begin{equation*}
            (1+\norm{\grad u}_{L^{\infty}})^{1/2}\tau_2^*(u,\gamma)<T_1,
        \end{equation*}
        then \eqref{e:CHE-exp-decay} still holds for every $t\geq 0$.
    \end{enumerate}
    \end{theorem}
    \item Advective Kuramoto-Sivashinsky equation (see \cite{feng2022global})
    \begin{equation}\label{e:KSE}
        \phi_t+u\cdot\grad\phi+\lap^2\phi=-\frac{1}{2}\abs{\grad\phi}^2-\lap\phi.
    \end{equation}
    \begin{theorem}
    \label{thm:Kuramoto-Sivashinsky equation}
    Let $d=2$, $u\in L^{\infty}((0,\infty);W^{1,\infty}(\T^2))$ and $\phi$ be the mild solution of \eqref{e:KSE} with initial data $\phi_0\in L^2(\T^2)$ (see Definition 2.1 in \cite{feng2022global} for the mild solution). Let $\norm{\phi_0-\bar{\phi}}_{L^2}=B>0$. There exists a time $T_2(B)$ such that if $\tau_2^*(u,1)<T_2(B)$, then \eqref{e:KSE} admits a global mild solution.
    \end{theorem}
    \item Advective thin-film equation (see \cite{feng2022suppression})
    \begin{equation}
\label{e:thin film eqn}
    \partial_t h + u\cdot\grad h + \lap^2 h = -\grad\cdot(\abs{\grad h}^{p-2}\grad h).
\end{equation}
\begin{theorem}
\label{thm:thin film equation}
For $2<p<3$, $u\in L^{\infty}\left([0,\infty), W^{1,\infty}(\T^d)\right)$, and $\mu>0$. Let $h$ be the mild solution of \eqref{e:thin film eqn} with initial data $h(0)=h_0\in L_0^2$. There exists a threshold value 
\begin{equation*}
T_1=T_1(\norm{h_0}_{L^2}, \mu, p)  
\end{equation*}
such that if 
\begin{equation*}\label{e:flow condition}
\left(\norm{u}_{L^\infty}\left(\tau_2^*(u,1)\right)^{\frac{5}{4}}+\left(\tau_2^*(u,1)\right)^{\frac{3}{4}}\right)\leqslant T_1(\norm{h_0}_{L^2}, \mu, p),
\end{equation*}
then there exists a constant $\beta>0$, such that for any $t>0$ it holds
\begin{equation*}
    \norm{h(t)}_{L^2}\leq\beta e^{-\mu t}\norm{h_0}_{L^2}.
\end{equation*}
\end{theorem}
\end{enumerate}

\begin{remark}
Theorem \ref{thm:main} can be used to show the existence of cellular flows satisfy the conditions for the two-dimensional Cahn-Hilliard equation (Theorem \ref{thm:Cahn-Hilliard} part (a)) and the Advective Kuramoto-Sivashinsky equation (Theorem \ref{thm:Kuramoto-Sivashinsky equation}). However, both three dimensional Cahn-Hilliard equation (Theorem \ref{thm:Cahn-Hilliard} part (b)) and the Advective thin-film equation (Theorem \ref{thm:thin film equation}) further requires the smallness of $\norm{u}_{H}^{\beta}\tau_2^*(u,1)$ for some $0<\beta<1$ and $H$ to be some Sobolev space. The existence of time-independent flow that satisfies such conditions is not clear and remains to be further investigated.

\end{remark}

\subsection{Applications to nonlinear diffusion equations}
In this section, we show how to generalize the dissipation enhancement results of cellular flows to nonlinear diffusion equations. 

Firstly, we consider the so-called \emph{porous medium equations} on $\T^2$:
\begin{equation}
\label{eqn:adv pme}
\partial_t\theta+u\cdot\nabla\theta-\nu \lap (\theta^{q})=0,\quad \theta(x,0)=\theta_0(x),
\end{equation}
where $q>1,\nu>0$, and $u$ is a divergence-free vector field. In \cite{ KiselevShterenbergEA08}, the authors showed that for time-periodic flows, if its unitary evolution operator (see equations (1.3)-(1.5) in \cite{ KiselevShterenbergEA08}) has no eigenfunctions in $H^1$, then the time-periodic flow is \emph{relaxation-enhancing} to \eqref{eqn:adv pme}. Our conclusion is similar to that in \cite{ KiselevShterenbergEA08}, but we consider $u$ is a time independent cellular flow. Consider the same class of initial data as in  \cite{ KiselevShterenbergEA08}, that is: $0<h\leq\theta_0\leq h^{-1}$. Then the classical theory (see, e.g., \cite{Vazquez2006}) yields there exists a unique solution correspond to $\theta_0$ provided $u\in C^{\infty}(\T^2)$. Further more, the maximum principle guarantees $h\leq \theta\leq h^{-1}$. Also observe that the average $\bar{\theta}=\bar{\theta}_0$ is still preserved by \eqref{eqn:adv pme}. We have following dissipation enhancement kind result for cellular flows.
\begin{theorem}
\label{thm:RH for pme}
Consider equation \eqref{eqn:adv pme} with $0<h\leq\theta_0\leq h^{-1}$ and $\norm{\theta_0-\bar{\theta}_0}_{L^2}=1$. Then for any $\tau>0$, there exist cellular flows with proper cell size and amplitude (depend on $\tau$, $h$, and $\nu$) such that  
\begin{equation}
    \norm{\theta(x,\tau)-\bar{\theta}}_{L^2}\leq\frac{1}{2}.
\end{equation}
\end{theorem}
\begin{remark}
Note that to compare with the linear case (that is, Theorem \ref{thm:main}), the main difference is that the cellular flows in Theorem \ref{thm:RH for pme} depend on the initial data $\theta_0$, or more precisely on $h$.
\end{remark}
\begin{proof}
Firstly, observe that for any $0<h\leq\psi\leq h^{-1}$, the expression $\int_{\T^2}\psi^{q-1}\abs{\grad\psi}^2dx$ is equivalent to the $H^1(\T^2)$ norm. Then the proof is almost the same as the linear case, except we need to establish a parallel estimate for Lemma \ref{lem:difference estimate}. For any cellular flow with cell size $m$: $u_m=u(mx)$, let $\theta_m^{\nu}$ denote the corresponding solution to \eqref{eqn:adv pme} and $\theta_{m}^0$ denotes the solution to the inviscid problem as in Lemma \ref{lem:difference estimate}. The authors in \cite{ KiselevShterenbergEA08} showed that 
\begin{align*}
 &\frac{d}{dt}\norm{\theta_m^{\nu}(t)-\theta_m^0(t)}_{L^2}^2
 \leq 2\nu\int_{\T^2}\lap (\theta_m^{\nu})^q(\theta_m^{\nu}-\theta_m^0)dx\\
 &\leq 2\nu q \left(\int_{\T^2}(\theta_m^{\nu})^{q-1}\abs{\grad\theta_m^{\nu}}^2 dx\right)^{1/2}\left(\int_{\T^2}(\theta_m^{\nu})^{q-1}\abs{\grad \theta_m^0}^2 dx\right)^{1/2}\\
 &\qquad-2\nu q \int_{\T^2}(\theta_m^{\nu})^{q-1}\abs{\grad\theta_m^{\nu}}^2 dx\\
 &\leq \frac{\nu q}{2}\int_{\T^2}(\theta_m^{\nu})^{q-1}\abs{\grad\theta_m^0}^2 dx\\
 &\leq\frac{\nu q h^{1-q}}{2} F_m^2(t)\norm{\grad\theta_0}_{L^2}^2.
\end{align*}
Here $F_m(t)$ is as in Lemma \ref{lem:difference estimate}. The remaining proof is almost the same as the linear case, only with some estimates may depend on $h$, so we omit them.
\end{proof}

The second nonlinear diffusion equation we consider is the advective p-Laplacian equation: 
\begin{equation}
\label{e:adv p-laplacian}
\partial_t\vartheta+u\cdot\nabla\vartheta-\nu \grad \cdot \left(\abs{\grad \vartheta }^{p-2} \grad \vartheta \right)=0,\quad \vartheta(x,0)=\vartheta_0(x),
\end{equation}
where $p>2$ and $\nu>0$. In the recent paper \cite{feng2021dissipation}, the author studied dissipation enhancement effect of time-dependent mixing flows to the weak solutions of \eqref{e:adv p-laplacian}. For simplicity, consider $\vartheta_0\in L_0^2(\T^2)$, then \eqref{e:adv p-laplacian} possesses a unique weak solution $\vartheta\in L_0^2(\T^2)$ (see Section 3 in \cite{feng2021dissipation}). With the same spirit as the \emph{porous medium equation}, the result in \cite{feng2021dissipation} can be easily generalized to the cellular flow scenario. Let $u$ be a cellular flow, we have:
\begin{theorem}
\label{thm:RH for plap}
Consider equation \eqref{e:adv p-laplacian} with $\vartheta_0\in L_0^2(\T^2)$ and $\norm{\vartheta_0}_{L^2}=1$. Then for any $\tau>0$, there exist cellular flows with proper cell size and amplitude (depend on $\tau,\norm{\vartheta_0}_{L^2}$, and $\nu$) such that  
\begin{equation}
    \norm{\vartheta(x,\tau)}_{L^2}\leq\frac{1}{2}.
\end{equation}
\end{theorem}
\begin{proof}
Similar to the \emph{porous medium case}, it is sufficient for us to observe the fact that 
\begin{equation}
\label{e:Lp and L2}
    \norm{\grad\psi}_{L^p}\geq C\norm{\grad\psi}_{L^2}
\end{equation}
for any $p\geq 2$, according to the H\"{o}lder inequality. Then we just need to re-estimate the difference between \eqref{e:adv p-laplacian} and the underlying inviscid problem. This estimate was established in Lemma 4.1 of \cite{feng2021dissipation}. For the convenience of readers, we cite the estimate here. For a cellular flow with cell size $m$, let $\vartheta_m^{\nu}$ denote the solution to \eqref{e:adv p-laplacian} and $\vartheta_m^{0}$ denote the inviscid problem. Then the estimate writes
\begin{equation}
\label{e:estimate between plap and transport}
    \norm{\vartheta_m^{\nu}(t)-\vartheta_m^{0}(t)}_{L^2}^2
    \leq \frac{C_p}{\norm{\grad u_m}_{L^{\infty}}}e^{2\norm{\grad u_m}_{L^{\infty}}t}\norm{\grad\varphi_0}_{L^p}^p,
\end{equation}
where $C_p$ is a constant depending on the parameter $p$. Combine the fact \eqref{e:Lp and L2} and estimate \eqref{e:estimate between plap and transport}, with the guide of \cite{feng2021dissipation} one can easily complete the proof of Theorem \ref{thm:RH for plap} by modifying the proof of the linear diffusion case.
\end{proof}

\section{Relaxation enhancing flows for hyper-diffusion}
In this chapter, we will construct a relaxation enhancing flow for hyper-diffusion but not for standard diffusion. 

To begin with, we recall that a number $\alpha\in\mathbb{R}$ is called $\beta$-Diophatine if there exists a constant $C$ such that for each $k\in\mathbb{Z}\setminus\{0\}$ we have 
$$\inf_{p\in \mathbb{Z}}|\alpha\cdot k+p|\geq\frac{C}{|k|^{1+\beta}}.$$
And the number $\alpha$ is Liouvillean if it is not Diophantine for any $\beta>0$. More specifically, for any $n(>1)$ and constant $C>0$, one can find $q_n\in\mathbb{Z}\setminus\{0\}$ and $p_n\in\mathbb{Z}$, such that 
\begin{equation}\label{eq:ln}
|\alpha\cdot q_n+p_n|\leq\frac{C}{|q_n|^{n}}.
\end{equation}
One can understand the Liouvillean numbers are the ones which can be very well approximated by rationals.
\begin{remark}
There exist lots of such irrational Liouvillean numbers. For example, a well-known one is given by $\alpha=\sum\limits_{n\geq 0}\frac{1}{10^{n!}}$.
\end{remark}
On the other hand, we denote by $\Phi^u_t$ the flow on the torus generated by $u$, and by $U^t$ the evolution operator on $L^2(\mathbb{T}^2)$
generated by $\Phi_t^u$:$(U^tf)(x)=f(\Phi_{-t}^u(x))$. By now, we are ready to state our result.
\begin{proposition}
\label{prop:RH flows}
There exists a smooth incompressible flow $u(x,y)$ on the two-dimensional torus so that the corresponding unitary evolution $U^t$ has a discrete spectrum on $H^1(\mathbb{T}^2)$ but none of the eigenfunctions of $U^t$
are in $H^2(\mathbb{T}^2)$.
\end{proposition}
To prove Proposition \ref{prop:RH flows}, we first prove following auxiliary lemma. Let $\mathbb{S}^1=[-\frac{1}{2},\frac{1}{2}]$ be the one-dimensional circle.
\begin{lemma}
\label{lem:H1noH2}
Consider the irrational Liouvillean number $\alpha$. There exists a $C^{\infty}(\mathbb{S}^1)$ mean-zero function $Q(\xi)$ so that the homology equation
\begin{equation}
\label{hom}
R(\xi+\alpha)-R(\xi)=Q(\xi)
\end{equation}
has a measurable solution $R(\xi):\mathbb{S}^1\rightarrow \mathbb{R}$ such that $R(x)$ is in $H^1(\mathbb{S}^1)$ but not in $H^{2}(\mathbb{S}^1)$. In addition, the function $R_{\lambda}(\xi)=e^{i\lambda R(\xi)}$ is in $H^1(\mathbb{S}^1)$ but not in $H^2(\mathbb{S}^1)$, for any $\lambda\in\mathbb{R}\setminus\{0\}$.
\end{lemma}

\begin{proof}
As the fist step, we construct a function $\tilde{R}$ in $L^2(\mathbb{S}^1)\cap L^4(\mathbb{S}^1)$ but not in $H^1(\mathbb{S}^1)$ satisfying \eqref{hom} with a mean zero function $\tilde{Q}(\xi)\in C^{\infty}(\mathbb{S}^1)$.
Let $\theta(t)$ be a $C^{\infty}$ bump function in $\mathbb{S}^1$, with $\theta(t)$ equals to $1$ when $-\frac{1}{16}\leq t\leq \frac{1}{16}$, and equals to $0$ when $\frac{1}{8}\leq t\leq \frac{1}{2}$ or $-\frac{1}{2}\leq t\leq -\frac{1}{8}$, and smooth everywhere with $|\frac{d\theta}{dt}|\leq 20$.

%({\color{red} not know, TBD, but we need $\theta$ to be smooth and only support at $t$ around $0$, hence $Q_q$ later only supported on an interval of size $q^{-\frac{a}{2}}$.})

% such that $\theta(t)=0$ for $|t|\leq\frac{1}{16}$, and $\theta(t)=1$ for $|t-\frac{3}{4}|\leq \frac{1}{16}$. In addtion, we can make $\theta(t)$ be nondecreasing for $t\in(0,\frac{3}{4})$, and nonincreasing for $t\in (\frac{3}{4},1)$, $\frac{d\theta}{dt}(\frac{1}{2})=2$, $0\leq\theta(t)\leq 1$, $|\frac{d\theta}{dt}|\leq 4$, and $|\frac{d^2\theta}{dt}|\leq 16$. 

Then for $q\geq 1$ we define $\tilde{Q}_q(x)=\theta(q^{6} x)$, for $x\in [-\frac{1}{8q^{6}},\frac{1}{8q^{6}}]$, and $0$ elsewhere in $[-\frac{1}{2},\frac{1}{2}]$. Note that $\tilde{Q}_q$ is still a $C^{\infty}$ function on $\mathbb{S}^1$. We define
\begin{equation}
\label{eqn:tildeQ}
\tilde{Q}=\sum_k (\tilde{Q}_{q^2_{k}}(\cdot-q_{k}\alpha)-\tilde{Q}_{q^2_{k}}),
\end{equation}
where the sequence $\{q_{k}\}$ is chosen as follows. By the definition of Liouvillean number $\alpha$ in \eqref{eq:ln}, fix any $C>0$, one can choose integers $q_{k}$ and $p_k$ so that $|\alpha\cdot q_k-p_k|\leq \frac{C}{|q_k|^{k}}$ for any $k>1$. Also note that since $\alpha$ is irrational, $q_k$ has infinitely many choices for each $k$. Therefore, without loss of generality, we can choose $q_k\geq k$ and $q_k$ increasing in $k$. \\
Now, we show $\tilde{Q}$ belongs to $C^{\infty}$ by checking the right-hand side of \eqref{eqn:tildeQ} converges very fast. By chain rule, we have $$||\tilde{Q}_q||_{H^r}\leq C(r) q^{6r}.$$ Then by using Fourier expansion, Parseval's identity and the definition of $\alpha$, we get
\begin{align*}
||\tilde{Q}_{q^2_k}(\cdot-q_{k}\alpha)-\tilde{Q}_{q^2_k}||_{H^r}&\leq C \||n|^r\hat{\tilde{Q}}_{q^2_k}(n)\left(e^{i 2\pi n\cdot q_k\alpha}-1\right)\|_{l^2}\\
&\leq C \||n|^r\hat{\tilde{Q}}_{q^2_k}(n)\left(e^{i 2\pi n\cdot (q_k\alpha-p_k)}-1\right)\|_{l^2}\\
&\leq C(r) \||n|^r\hat{\tilde{Q}}_{q^2_k}(n)\|_{l^2}/|q_k|^k\\
&\leq C(r)\frac{||\tilde{Q}_{q^2_k}||_{H^r}}{|q_{k}|^{k}}.
\end{align*}
By taking the $H^{r}$ norm on the both sides of \eqref{eqn:tildeQ}, and the series on the right-hand side is bounded by $C\cdot\sum\limits_k q_{k}^{-10}\leq C\cdot\sum\limits_k \frac{1}{k^{10}}<\infty$ for any $k\geq 12r+10$. Therefore, $\tilde{Q}$ belongs to $H^r$ for any $r$, which further yields $\tilde{Q}$ is smooth by the Sobolev embedding arguments.\\
Before defining the function $\tilde{R}$, we define $\tilde{R}_q=-\sum\limits_{l=0}^{q-1}\tilde{Q}_{q^2}(\cdot-l\alpha)$, then immediately we have $$\tilde{Q}_{q^2}(\cdot-q\alpha)-\tilde{Q}_{q^2}(\cdot)=\tilde{R}_q(\cdot-\alpha)-\tilde{R}_q(\cdot).$$ 

%Since we have $m(supp(R_q))\leq q m(supp(Q_{q^2}))\leq C q^{-2}$, so we have $\sum_q m(supp(R_q))<\infty$. Which means for any sequence $\{a_q\}$ $\sum_q a_q R_q$ converges in measure, so is a measurable function. 

Now we define $\tilde{R}=\sum_k \tilde{R}_{q_{k}}$, which satisfies $\tilde{R}(\cdot-\alpha)-\tilde{R}=\tilde{Q}$. Let $m$ be the standard Lebesgue measure. Since we have $m(supp(\tilde{R}_q))\leq q\cdot m(supp(\tilde{Q}_{q^2}))\leq \frac{1}{4} q^{-11}$, so we have $\sum_q m(supp(\tilde{R}_q))<\frac{1}{2}$. Which means for any sequence $\{a_q\}$, $\sum_q a_q \tilde{R}_q$ converges in measure, so is a measurable function in $\mathbb{S}^1$. Moreover, since we have $||\tilde{Q}_{q^2}||_{L^4}\leq C q^{-3}$, $||\tilde{R}_q||_{L^4}\leq q||\tilde{Q}_{q^2}||_{L^4}\leq C q^{-2}$, so $\tilde{R}$ is in $L^4$, hence also in $L^2$ and is finite almost everywhere. However, let $t\in[-\frac{1}{16},\frac{1}{16}]$, then for any positive integer $k_0$, $\tilde{R}_{q_k}(t\cdot\frac{1}{q^{12}_{k_0}})\leq -\tilde{Q}_{q^2_k}(t\cdot\frac{1}{q^{12}_{k_0}})=-\theta(\frac{q_k^{12}}{q_{k_0}^{12}}t)=-1$, for any $k\leq k_0$. This means $\tilde{R}(t\cdot\frac{1}{q^{12}_{k_0}})\leq -k_0$, for any large positive integer $k_0$ and any $t\in [-\frac{1}{16},\frac{1}{16}]$. This means for any $k_0>1$, there is a positive-measure set $\mathfrak{M}_{k_0}\subset \mathbb{S}^1$ such that $\tilde{R}\leq -k_0$ on it. Hence, there cannot be a continuous function $\bar{R}$ on $\mathbb{S}^1$ such that $\tilde{R}(x)=\bar{R}(x)$ almost everywhere, otherwise $\bar{R}(x)$ must be unbounded in $\mathbb{S}^1$, hence by Sobolev embedding $\tilde{R}$ is not in $H^1$. As a consequence, we construct such function $\tilde{R}$ in $L^2\cap L^4$ but not in $H^1$.

Now, without loss of generality, by subtracting constants and swap the sign, we can assume $\tilde{R}$ we get is mean zero, and together with $\tilde{Q}$ satisfying (\ref{hom}), with the same regularity. Let $R(x)=\int_0^x \tilde{R}(y)dy$, $Q(x)=\int_0^x\tilde{Q}(y)dy+\int_0^{\alpha}\tilde{R}(y)dy$, as $\tilde{R}(y+\alpha)-\tilde{R}(y)=\tilde{Q}(y)$, taking integration from $0$ to $x$, we have $R(x+\alpha)-R(x)=Q(x)$. It is clear that $Q$ is in $C^{\infty}(\mathbb{S}^1)$. By Fubini theorem, one can easily see that if the $n$-th Fourier coefficient of $\tilde{R}$ is $a_n$, then the $n$-th Fourier coefficient of $R$ is $\frac{a_n}{n}$ for $n\neq 0$. Hence, $R(x)$ is not in $H^2(\mathbb{S}^1)$. Meanwhile, by Lebesgue differentiation theorem, $R(x)$ is in $H^1(\mathbb{S}^1)\cap W^{1,4}(\mathbb{S}^1)$. By chain rule and the fact that $R$ is in $W^{1,4}(\mathbb{S})\cap H^1({\mathbb{S}})$, we also have $R_{\lambda}(x)=e^{i\lambda R(x)}$ is in $H^1$ but not in $H^2$, for $\lambda\neq 0$.
\end{proof}
The remaining proof of Proposition \ref{prop:RH flows} is identically the same as the proof of Proposition 6.2 in \cite{ConstantinKiselevEA08} (but replace Proposition 6.3 in \cite{ConstantinKiselevEA08} by Lemma \ref{lem:H1noH2} above). For completeness, we summarize the idea of the proof in the following two propostions.
\begin{proposition}\label{prop:defineF}
There exists a function $F(x,y)$ on $\mathbb{T}^2$, such that the unitary evolution to the flow $\tilde{u}(x,y)=(\frac{\alpha}{F(x,y)},\frac{1}{F(x,y)})$ has a discrete spectrum on $H^1(\mathbb{T}^2)$ but none of the eigenfunctions of such evolution are in $H^2(\mathbb{T}^2)$.
\end{proposition}
\begin{remark}
The construction of $F$ is explicit. For example, as in \cite{ConstantinKiselevEA08}, by adding a constant and rescaling, one can replace the right hand side of \eqref{hom} by $Q-1$ instead of $Q$, such that $Q$ is positive and has mean $1$, and 
$$
F(x,y)=b+\psi(y)(Q(x-\alpha y)-b),
$$
where $b=\frac{1}{2}\min Q$, $\psi(y)$ is a smooth cutoff function supported near $0$. And we also emphasize that $\tilde{u}$ may not be incompressible.
\end{remark}
\begin{proposition}\label{prop:conjugate}
For $F(x,y)$ and the corresponding $\tilde{u}(x,y)$ in Proposition \ref{prop:defineF}, there exists an incompressible flow $u(x,y)$ and a measure-preserving map $S$ from $\mathbb{T}^2$ to $\mathbb{T}^2$, such that the evolution of $u$ is conjugate by $S$ to the evolution of $\tilde{u}$, and the evolution of $u$ possesses the same spectrum as that of $\tilde{u}$.
\end{proposition}
From Proposition \ref{prop:conjugate} one can directly get Proposition \ref{prop:RH flows}.

\subsection*{Acknowledgments}
The work of Y.F. is supported by the National Key R\&D Program of China, Project Number 2021YFA1001200. The work of X.X. is supported by the National Key R\&D Program of China, Project Number 2021YFA1001200, and the NSFC Youth program, grant number 12101278. And the authors would like to thank Andrej Zlato\v{s}, Gautam Iyer, and Jean-Luc Thiffeault for helpful suggestions.

\appendix
\section{Proof of Theorem \ref{thm:main}}
Similar to \eqref{e:eform1.1}, we rescale the time variable of our equation by $A$ and call $\epsilon=\frac{1}{A}$. Then it is sufficient for us to show that given $\tau>0$, there exists $m_0$ and corresponding $\epsilon_0=\epsilon_0(m_0)>0$ such that if $m>m_0$ and $\epsilon<\epsilon_0$, we have $\norm{\theta_m^{\epsilon}(\tau/\epsilon)}_{L^2}<\frac{1}{2}$ for any initial datum $\theta_0$ with $\norm{\theta_0}_{L^2}=1$. Notice that $(-\lap)^{\alpha}$ is an unbounded positive operator with a discrete spectrum. Its eigenvalues $\Lambda_n=\lambda_n^{\alpha}\rightarrow\infty$ as $n\rightarrow\infty$, here $\lambda_n$ stands for the eigenvalue of $-\lap$. Thus we can choose $N$ large enough, so that $e^{-\Lambda_N \tau/80}<\frac{1}{2}$. Define the sets $K=\left\{\theta\in S\big\vert\norm{\theta}_{\dot{H}^{\alpha}}^2\leqslant\Lambda_N\right\}\subset S$ and $K_1=\left\{\theta\in K\big\vert\norm{P_p \theta}_{L^2}^2\geqslant 1/2\right\}$, then $K$ is compact. According to Lemma \ref{lem:large_H1}, there exists $m_0$ large enough so that for any fixed $m>m_0$, we can further choose $J\geqslant N$ and $J\geqslant N_p(5\Lambda_N,K,m)$. Define
\begin{equation}
\label{e:tau_1}
    \tau_1=\max\left\{T_p(5\Lambda_N,K,m),T_c(J,\frac{\Lambda_N}{20\Lambda_J},m,K)\right\},
\end{equation}
where $T_p$ is from Lemma \ref{lem:H1 mean large}, and $T_c$ from Lemma \ref{lem:RAGE thm}. Notice that $\tau_1$ only depends on the value of $m$ and $\tau$. Finally, choose $\epsilon_0>0$ small enough (may depend on $m$, $\tau$) so that $\tau_1<\tau/2\epsilon_0$, and
\begin{equation}
    \epsilon_0\int_0^{\tau_1}F_m^2(t)dt\leqslant\frac{1}{20\Lambda_J},
\end{equation}
where $F_m(t)$ is the function from condition \eqref{e:F_m condition}.
Take any $\epsilon<\epsilon_0$. If $\norm{\theta_m^{\epsilon}(s)}_{\dot{H}^{\alpha}}^2\geqslant\Lambda_{N}\norm{\theta_m^{\epsilon}(s)}_{L^2}^2$ holds for all $s\in[0,\tau/\epsilon]$, then according to Lemma \ref{lem:trivial case} and the choice of $\Lambda_N$ we directly get $\norm{\theta_m^{\epsilon}(\tau/\epsilon)}_{L^2}\leqslant e^{-2\Lambda_{N}\tau}\leqslant\frac{1}{2}$, and we are done. Otherwise, let $\tau_0$ be the first time in $[0,\tau/\epsilon]$ such that $\norm{\theta_m^{\epsilon}(\tau_0)}_{\dot{H}^{\alpha}}^2\leqslant\Lambda_{N}\norm{\theta_m^{\epsilon}(\tau_0)}_{L^2}^2$. Then we claim that the following decay holds on the time interval $[\tau_0,\tau_0+\tau_1]$:
\begin{equation}
\label{e:the claim}
    \norm{\theta_m^{\epsilon}(\tau_0+\tau_1)}_{L^2}^2\leqslant e^{-\Lambda_N \epsilon\tau_1/20}\norm{\theta_m^{\epsilon}(\tau_0)}_{L^2}^2.
\end{equation}
For simplicity, denote $\theta_m^{\epsilon}(\tau_0)=\theta_0$. Consider the two equations in \eqref{e:two equations}, then by Lemma \ref{lem:difference estimate} and the choice of $\epsilon_0$, we have
\begin{equation}
\label{e:difference estimate}
    \norm{\theta_m^{\epsilon}(t)-\theta_m^0(t)}_{L^2}^2\leqslant\frac{\Lambda_{N}}{40\Lambda_{J}}\norm{\theta_0}_{L^2}^2
\end{equation}
for all $t\in[\tau_0,\tau_0+\tau_1]$. Split $\theta_m^0(t)=\theta_{m,c}(t)+\theta_{m,p}(t)$, observe that $\theta_{m,c}$ and $\theta_{m,p}$ also solve the inviscid problem $\frac{d}{dt}\theta(t)=i\FL_m\theta$, but with initial data $P_c\theta_0$ and $P_p\theta_0$ at $t=\tau_0$, respectively. Now, we consider two cases:
\\
\textbf{Case I.} Assume that $\norm{P_c\theta_0}_{L^2}^2\geqslant\frac{3}{4}\norm{\theta_0}^2$, equivalently, $\norm{P_p\theta_0}_{L^2}^2\leqslant\frac{1}{4}\norm{\theta_0}^2$. Note that since $\theta_0/\norm{\theta}_{L^2}\in K$, thus we can apply Lemma \ref{lem:RAGE thm}, by our choice of $\tau_1$ in \eqref{e:tau_1}, we get
\begin{equation}
\label{e:estimate by RAGE}
    \frac{1}{\tau_1}\int_{\tau_0}^{\tau_0+\tau_1}\norm{P_N\theta_{m,c}(t)}_{L^2}^2dt\leqslant\frac{\Lambda_N}{20\Lambda_{J}}\norm{\theta_0}_{L^2}^2.
\end{equation}
By elementary inequalities, we further have
\begin{align}
    \norm{(I-P_N)\theta_m^0(t)}_{L^2}^2
    &\geqslant\frac{1}{2}\norm{(I-P_N)\theta_{m,c}(t)}_{L^2}^2-\norm{(I-P_N)\theta_{m,p}(t)}_{L^2}^2\\
    &\geqslant\frac{1}{2}\norm{\theta_{m,c}(t)}_{L^2}^2-\frac{1}{2}\norm{P_N\theta_{m,c}(t)}_{L^2}^2-\norm{\theta_{m,p}(t)}_{L^2}^2.
\end{align}
Combine the facts of the free evolution $e^{i\FL_m t}$ is unitary, $\Lambda_J\geqslant\Lambda_N$, assumptions on $\norm{P_{c}\theta_0}_{L^2}$ and $\norm{P_{p}\theta_0}_{L^2}$, and the estimate \eqref{e:estimate by RAGE}, we have
\begin{equation}
    \frac{1}{\tau_1}\int_{\tau_0}^{\tau_0+\tau_1}\norm{(I-P_N)\theta_m^0(t)}_{L^2}^2 dt\geqslant\frac{1}{10}\norm{\theta_0}_{L^2}^2.
\end{equation}
By using the estimate \eqref{e:difference estimate}, we get
\begin{equation}
      \frac{1}{\tau_1}\int_{\tau_0}^{\tau_0+\tau_1}\norm{(I-P_N)\theta_m^{\epsilon}(t)}_{L^2}^2 dt\geqslant\frac{1}{40}\norm{\theta_0}_{L^2}^2.
\end{equation}
The above estimate further implies 
\begin{equation}
    \int_{\tau_0}^{\tau_0+\tau_1}\norm{\theta_m^{\epsilon}(t)}_{\dot{H}^{\alpha}}^2 dt\geqslant\frac{\Lambda_{J}\tau_1}{40}\norm{\theta_0}_{L^2}^2.
\end{equation}
Combine the above inequality with identity \eqref{e:L2 rate} to get
\begin{equation}
    \norm{\theta_m^{\epsilon}(\tau_0+\tau_1)}_{L^2}^2\leqslant \left(1-\frac{\Lambda_{J}\epsilon\tau_1}{20}\right)\norm{\theta_m^{\epsilon}(\tau_0)}_{L^2}^2\leqslant e^{-\Lambda_{J}\epsilon\tau_1/20}\norm{\theta_m^{\epsilon}(\tau_0)}_{L^2}^2.
\end{equation}
This completes the proof of \eqref{e:the claim} in the first case, where we used the fact that $\Lambda_J\geqslant\Lambda_N$.\\
\textbf{Case II.} Now we assume $\norm{P_p\theta_0}_{L^2}^2\geqslant\frac{1}{4}\norm{\theta_0}_{L^2}^2$. In this case $\theta_0/\norm{\theta_0}_{L^2}\in K_1$, so that we can apply Lemma \ref{lem:H1 mean large}. Thus, by the choice of $J$ and $\tau_1$,
\begin{equation}
\label{e:H1 estimate1}
    \frac{1}{\tau_1}\int_{\tau_0}^{\tau_0+\tau_1}\norm{P_N\theta_{m,p}(t)}_{\dot{H}^{\alpha}}^2dt\geqslant 5\Lambda_{N}\norm{\theta_0}_{L^2}^2.
\end{equation}
Notice that \eqref{e:estimate by RAGE} still holds because of the choice of $\tau_0$ and $\tau_1$, it yields 
\begin{equation}
\label{e:H1 estimate2}
    \frac{1}{\tau_1}\int_{\tau_0}^{\tau_0+\tau_1}\norm{P_N\theta_{m,c}(t)}_{\dot{H}^{\alpha}}^2 dt\leqslant \frac{\Lambda_{N}}{20}\norm{\theta_0}_{L^2}^2.    
\end{equation}
Combine the estimates \eqref{e:H1 estimate1} and \eqref{e:H1 estimate2} we get
\begin{equation}
\label{e:H1 estimate3}
    \frac{1}{\tau_1}\int_{\tau_0}^{\tau_0+\tau_1}\norm{P_N\theta_m^0(t)}_{\dot{H}^{\alpha}}^2 dt\geqslant 2\Lambda_N\norm{\theta_0}_{L^2}^2.
\end{equation}
Then \eqref{e:difference estimate} and \eqref{e:H1 estimate3} give
\begin{equation}
\label{e:viscous H1 lower bd}
    \int_{\tau_0}^{\tau_0+\tau_1}\norm{P_N\theta_m^{\epsilon}(t)}_{\dot{H}^{\alpha}}^2 dt\geqslant \frac{\Lambda_N\tau_1}{2}\norm{\theta_0}_{L^2}^2
\end{equation}
where we also used the fact that $\norm{P_N\theta_m^{\epsilon}-P_N\theta_m^0}_{\dot{H}^{\alpha}}^2\leqslant\Lambda_{J}\norm{\theta_m^{\epsilon}-\theta_m^0}_{L^2}^2$. Then similar to the previous case, \eqref{e:viscous H1 lower bd} implies
\begin{equation}
    \norm{\theta_m^{\epsilon}(\tau_0+\tau_1)}_{L^2}^2\leqslant e^{-\Lambda_N \epsilon\tau_1}\norm{\theta_m^{\epsilon}(\tau_0)}_{L^2}^2,
\end{equation}
which completes the proof of the claim \eqref{e:the claim} in the second case.

Finally, we summarize all the cases. Firstly, we see that if $\norm{\theta_m^{\epsilon}(\tau_0)}_{\dot{H}^{\alpha}}^2\leqslant\Lambda_{N}\norm{\theta_m^{\epsilon}(\tau_0)}_{L^2}^2$, then the decay \eqref{e:the claim} always holds. On the other hand, for any time interval $I=[a,b]$ such that $\norm{\theta_m^{\epsilon}(t)}_{\dot{H}^{\alpha}}^2\geqslant\Lambda_{N}\norm{\theta_m^{\epsilon}(t)}_{L^2}^2$ on $I$, then Lemma \ref{lem:trivial case} yields
\begin{equation}
    \norm{\theta_m^{\epsilon}(b)}_{L^2}^2\leqslant e^{-2\Lambda_N\epsilon(b-a)} \norm{\theta_m^{\epsilon}(a)}_{L^2}^2.
\end{equation}
Combine the two decay factors obtained from \eqref{e:the claim} and the inequality above, also use the the fact $\tau_1<\tau/2\epsilon$, then we can further find $\tau_2\in[\tau/2\epsilon,\tau/\epsilon]$ such that
\begin{equation}
    \norm{\theta_m^{\epsilon}(\tau_2)}_{L^2}^2\leqslant e^{-\Lambda_N\epsilon\tau_2/20}\leqslant e^{-\Lambda_N\tau/40}<1/4
\end{equation}
by the choice of $\Lambda_N$. Then by the monotonicity of the $L^2$ norm \eqref{e:L2 rate} we have $\norm{\theta_m^{\epsilon}(\tau/\epsilon)}_{L^2}\leqslant\norm{\theta_m^{\epsilon}(\tau_2)}_{L^2}<1/2$, this completes the proof. 
\bibliographystyle{plain}
\bibliography{reference}

\end{document}